\documentclass[runningheads,envcountsame,numbook]{svjour2-}
\smartqed  
\usepackage{graphicx}
\usepackage{mathptmx}      
\usepackage{amsmath,amssymb,latexsym,xcolor}
\newcommand{\Aut}[1]{\mathsf{Aut}(#1)}%

\newcommand{\Hom}{\operatorname{Hom}}%

\renewcommand{\emph}{\textbf}

\newcommand{\normaleq}{\;\;{\raisebox{.198ex}{$\vartriangleleft$}
    \mkern-18.5mu\leqslant}\;\;}%
\journalname{%
{\color{blue} Submitted version}, 
compiled~{\color{blue} June 15, 2008}  
in {\color{blue}Rheinfelden, Germany}
}
\begin{document}

\title{Geometric characterization 
of flat groups of automorphisms 
\thanks{Supported by Australian Research Council grant DP0556017.}
}


\author{
U. Baumgartner     \and
              G. Schlichting \and 
              G. Willis
}


\institute{Udo Baumgartner\at
              School of Mathematical and Physical Sciences, 
              The University of Newcastle, University Drive, Building V, 
              Callaghan, NSW 2308, Australia; 
              \email{Udo.Baumgartner@newcastle.edu.au}           
           \and
                G\"unter Schlichting\at 
	     Technische Universit\"at M\"unchen, 
	     Center of Mathematical Sciences, M8, 
	     D-85748 Garching bei M\"unchen, 
	     Germany; 
	      \email{schlicht@mathematik.tu-muenchen.de}
	  \and 
	     George A. Willis\at
              School of Mathematical and Physical Sciences, 
              The University of Newcastle, University Drive, Building V, 
              Callaghan, NSW 2308, Australia; 
              \email{George.Willis@newcastle.edu.au}           
}

\date{Received: date / Revised: date}
\maketitle

\begin{abstract}
If 
$\mathcal{H}$ is 
a flat group 
of automorphisms 
of finite rank $n$ 
of a totally disconnected, locally compact group $G$, 
then 
each orbit 
of $\mathcal{H}$ 
in the metric space $\mathcal{B}(G)$ 
of compact, open subgroups 
of $G$ 
is quasi-isometric to 
$n$-dimensional euclidean space. 
In this note 
we prove 
the following partial converse: 
Assume that 
$G$ is 
a totally disconnected, locally compact group 
such that 
$\mathcal{B}(G)$ is 
a proper metric space 
and 
let $\mathcal{H}$ be 
a group of automorphisms 
of $G$ 
such that 
some (equivalently every) orbit 
of $\mathcal{H}$ 
in $\mathcal{B}(G)$ 
is quasi-isometric to 
$n$-dimensional euclidean space, 
then 
$\mathcal{H}$ 
has a finite index subgroup 
which is flat 
of rank $n$. 
We can 
draw this conclusion 
under weaker assumptions. 
We also 
single out 
a naturally defined 
flat subgroup 
of such 
groups 
of automorphisms. 
\subclass{22D05 \and 22D45 \and 20E25 \and 20E36}
\end{abstract}

\section{Introduction}\label{sec:intro}

In this article, 
all automorphisms 
of topological groups 
are assumed to be 
continuous 
with continuous inverse. 
Flat groups of 
automorphisms of 
a totally disconnected, locally compact group 
were introducted 
and studied 
in~\cite{tidy<:commAut(tdlcG)}. 
A group 
of automorphisms, 
$\mathcal{H}$ say, 
of 
a totally disconnected, locally compact group $G$ 
is flat,  
if 
there exists 
a compact, open subgroup, 
$O$ say, 
of $G$, 
called \emph{minimizing for $\mathcal{H}$}, 
that minimizes 
all the displacement functions 
of elements 
in $\mathcal{H}$ 
on the metric space, 
$(\mathcal{B}(G),d)$,  
of compact, open subgroups 
of $G$ 
relative to 
the distance function 
$d(V,W):=\log\bigl(|V\colon V\cap W|\cdot |W\colon W\cap V|\bigr)$. 
One can 
rephrase 
this condition 
by saying that 
for every $\varphi\in\mathcal{H}$ 
the integer 
$|\varphi(O) \colon \varphi(O)\cap O|$ 
attains 
the value of 
the \emph{scale function}, 
$s_G$, 
at $\varphi$, 
where 
$s_G$ 
is defined by 
$s_G(\varphi) := \min \{|\varphi(V ) \colon \varphi(V)\cap V| \colon V\in \mathcal{B}(G)\}$ .

Flat groups 
are, 
in a sense, 
generalizations of 
split tori 
in semisimple algebraic groups 
over nonarchimedian local fields. 
They admit 
a dynamically defined 
`root system' $\Phi_\mathcal{H}$, 
which governs 
the decompsition 
of 
any minimizing subgroup 
for $\mathcal{H}$
into 
a product of 
associated `eigenfactors'  
on which 
every element 
of $\mathcal{H}$ 
is either 
expanding or contracting. 

In this article 
we study 
the question 
of the extent 
to which 
flat groups 
are characterized by 
their geometric properties. 
These properties 
are 
summarized in 
the following theorem.

\begin{theorem}[geometric properties of flat groups of automorphisms]
\label{thm:norm>flat-gp}
~\\
Let $\mathcal{H}$ be 
a flat group of 
automorphisms of 
a totally disconnected, locally compact group $G$ 
and 
let $O$ be minimizing for $\mathcal{H}$. 
Put $\mathcal{H}(1):=\{\varphi\in\mathcal{H}\colon s_G(\varphi)=1=s_G(\varphi^{-1})\}$.
Then 
\begin{enumerate}
\item 
$\mathcal{H}_O=\mathcal{H}(1)\normaleq \mathcal{H}$ 
and 
$\mathcal{H}/\mathcal{H}(1)$ is a free abelian group, 
whose $\mathbb{Z}$-rank 
is called 
the \emph{rank} 
of $\mathcal{H}$; 
\item
there is 
a set 
$\Phi_\mathcal{H}\subseteq \Hom(\mathcal{H},\mathbb{Z})$ 
of surjective homomorphisms 
(which is independent 
of $O$) 
such that 
$\bigcap_{\rho\in\Phi_\mathcal{H}} \ker(\rho)=\mathcal{H}(1)$ 
and 
such that 
for each $\varphi\in \mathcal{H}$ 
we have 
$\rho(\varphi)\neq 0$ for only finitely many $\rho$ in $\Phi_\mathcal{H}$; 
\item 
there are 
positive integers $t_\rho$ for $\rho\in \Phi_\mathcal{H}$ 
(which are independent 
of $O$) 
such that 
the function $\|\cdot\|_\mathcal{H}$ 
defined on $\mathcal{H}/\mathcal{H}(1)$ 
by 
the rule 
$\|\varphi \mathcal{H}(1))\|_\mathcal{H}:=d(\varphi(O),O)$ 
takes the value 
$
\sum_{\rho\in\Phi_\mathcal{H}} \log(t_\rho)\,|\rho(\varphi)|
$
at $\varphi\mathcal{H}(1)$ 
for $\varphi\in\mathcal{H}$. 
\end{enumerate}
Hence 
the function $\|\cdot\|_\mathcal{H}$ 
extends to a norm on the vector space $\mathbb{R}\otimes \mathcal{H}/\mathcal{H}(1)$ 
and 
each orbit 
of a flat group of automorphisms 
of finite rank 
in $\mathcal{B}(G)$ 
is quasi-isometric to $\mathbb{R}\otimes \mathcal{H}/\mathcal{H}(1)$.  
\end{theorem}

Theorem~\ref{thm:norm>flat-gp} 
follows from 
results 
proven in 
the paper~\cite{tidy<:commAut(tdlcG)}; 
the reader 
may 
consult 
the survey article~\cite{tdlcGs-geomObjs} 
for a compact exposition, 
with 
only short hints 
of proofs; 
the statement 
on all orbits 
of $\mathcal{H}$ 
in $\mathcal{B}(G)$ 
is obtained 
by applying 
the following lemma 
with $\mathbf{B}:=\mathcal{B}(G)$ 
and 
the natural action of $\mathcal{H}$ 
on $\mathbf{B}$. 
The proof 
of the lemma 
is left to 
the reader. 

\begin{lemma}[distances on orbits under isometric actions differ by a constant]
\label{lem:comparison(orbits(Gs(Isoms)))}
Let 
a group $\mathcal{H}$ 
act by isometries on a metric space $\mathbf{B}$.  
Then 
any two orbits 
of $\mathcal{H}$ 
in $\mathbf{B}$ 
are $(1,\epsilon)$-quasi-isometric, 
with $\epsilon$ only depending on 
the pair of orbits.   
In particular 
all orbits 
of $\mathcal{H}$ 
in $\mathbf{B}$ 
have the same 
type of growth. 
\end{lemma}

In this paper  
we 
address the question 
whether 
flat groups 
are 
virtually 
characterized by 
the rough isometry class 
of their orbits 
in 
the space 
of compact, open subgroups; 
more precisely, 
we study 
the following problem. 

\begin{problem}[question whether flatness can be characterized geometrically]
\label{Baumgartner_2}
Let $G$ be 
a totally disconnected, locally compact group 
and 
let $\mathcal{H}$ 
be a group of automorphisms 
of $G$ 
such that some (equivalently, any) 
orbit of $\mathcal{H}$ 
in the metric space 
of compact, open subgroups 
of $G$ 
is quasi-isometric 
to $\mathbb{R}^n$ 
for some $n$ in $\mathbb{N}$. 
Does it follow that 
$\mathcal{H}$ 
has a finite index subgroup 
which is 
a flat group of automorphisms of $G$ 
of rank $n$? 
\end{problem} 

Under the conditions 
described 
in Problem~\ref{Baumgartner_2},  
the group $\mathcal{H}$ itself 
need not be 
flat; 
two examples 
illustrating this 
will be given 
in Section~\ref{sec:examples&applications}. 

A geometric characterization 
as formulated 
in Problem~\ref{Baumgartner_2} 
can be expected 
to be 
useful 
in the detection 
and classification 
of flat subgroups 
of totally disconnected, locally compact groups 
that are 
automorphism groups 
of a geometric structure 
because 
for such groups 
the space 
of compact, open subgroups 
will often be 
related to 
that geometric structure. 
An example 
is presented 
in Theorem~\ref{thm:max-flat<Gs(semisimpleGs)}.

While 
we can not solve 
Problem~\ref{Baumgartner_2} 
in full generality, 
Theorem~\ref{thm:MainThm_weak-form} 
below 
solves it 
affirmatively 
in the important case 
where 
the metric space 
of compact, open subgroups 
is proper; 
we are 
actually 
able to obtain 
the same conclusion 
as 
in Theorem~\ref{thm:MainThm_weak-form} 
under weaker hypotheses, 
see Theorem~\ref{thm:MainThm}. 

We also have 
quite strong control 
over the flat group 
of finite index 
whose existence 
we guarantee; 
compare 
Theorem~\ref{thm:bounded-conjugacy-classe=max-normal-flat}. 
However, 
there is 
probably 
no way 
to describe 
a minimizing subgroup 
for that flat subgroup 
of finite index. 
Abstract 
flatness criteria 
that do not 
effectively produce 
a minimizing subgroup 
have been proved 
before: 
a purely algebraic criterion 
by Willis 
in~\cite{tidy<:commAut(tdlcG)} 
and 
several 
`algebraic-bounded' criteria 
by Shalom and Willis 
in~\cite{polycyclicGs=virt-flat}, 
one of which, 
restated here 
as Theorem~\ref{thm:nilpotent=flat} 
we use 
in the proof 
of Theorems~\ref{thm:MainThm_weak-form} 
and~\ref{thm:MainThm}.

\begin{theorem}[geometric characterization of flatness in the proper case]
\label{thm:MainThm_weak-form}
~\\Let $G$ be 
a totally disconnected, locally compact group 
and 
let $\mathcal{H}$ 
be a group of automorphisms 
of $G$ 
such that some (equivalently, any) 
orbit of $\mathcal{H}$ 
in the metric space 
of compact, open subgroups 
of $G$ 
is quasi-isometric 
to $\mathbb{R}^n$ 
for some $n$ in $\mathbb{N}$. 
Assume that 
the metric space $\mathcal{B}(G)$ 
of compact, open subgroups 
of $G$ 
is proper. 
Then
$\mathcal{H}$ 
has a finite index subgroup 
which is 
a flat group of automorphisms of $G$ 
of rank $n$. 
\end{theorem} 

We now 
discuss 
known results 
that are 
similar to 
and, 
it turns out, 
also 
related to 
Theorem~\ref{thm:MainThm_weak-form}. 
These results 
 illustrate 
why 
one might expect 
to obtain 
such a result 
in the first place. 

If 
we ignore 
the special form of 
the group $\mathcal{H}$ 
that acts 
on the quasi-flat orbits,  
what 
we seek 
may be called 
a `quasi-isometric version' 
of Bieberbach's First Theorem 
on space groups. 
(A statement 
of Bieberbach's First Theorem 
can be found 
e.g. 
in~\cite[Theorem~1$'$]{BieberbachThm(SGs)+disc-unif<Gs(LieGs)} 
or 
in~\cite[Bieberbach Theorem~1]{account(Thry(crystallogGs))} 
and 
in the books 
\cite{found(DiffGeom)1,almostBieberbachGs,Geom(non-pos-curvMfds),BieberbachGs+flatMfds};  
for the original articles 
see~\cite{BewegG(EuklidR),BewegG(EuklidR)2}.) 

We are not aware of 
the existence of 
such a 
`quasi-isometric version' 
of Bieberbach's First Theorem,  
but 
several 
well-known results 
have similar conclusions.  
For example,  
by Th\'eor\`eme~17 in~\cite{PM83} 
$\mathbb{Z}^n$ is quasi-isometrically rigid 
for every $n$, 
that is, 
a finitely generated group 
that is 
quasi-isometric to 
a free abelian group 
of rank~$n$ 
has a subgroup 
of finite index 
that is 
a free abelian group 
of rank~$n$. 
The latter result 
relied 
until Shalom's paper~\cite{harmAna+Cohom+large-scaleGeom(amenGs)} 
(Theorem~1.1, proven on p. 126; 
see also 
Corollary~1.5 in~\cite{isom-acts(HilbertSs)_growth(cocycles)}, proven in \S4.3),  
on 
Gromov's characterization of 
finitely generated groups 
of polynomial growth, 
which 
will 
also 
play 
a prominent role 
in this paper. 

A graph-theoretic 
analogue 
of Bieberbach's First Theorem 
may be seen 
in 
Theorem~1  
of Trofimov's paper~\cite{Gphs_poly-growth} 
(cited as Theorem~2.1 
in the survey paper~\cite{survey>gphs(poly-growth)}). 
Trofimov's work~\cite{Gphs_poly-growth} 
implies 
Bieberbach's First and Third Theorem 
as noted 
on page~417 
of that article 
and 
also implies 
Gromov's characterization of 
finitely generated groups 
of polynomial growth 
(but note that 
the main result 
in the paper~\cite{{Gs(poly-growth)+expand-maps}} 
that shows 
Gromov's result 
is used in 
Trofimov's work). 
In 
the graph-theoretic 
context 
the papers~\cite{fin-contr(Gphs_poly-growth),Gs_act>Gphs_poly-growth}  
also 
discuss 
interesting 
aspects. 
The above examples 
and 
our own success 
in proving 
an analogue 
of Bieberbach's First Theorem 
in our, 
admittedly restricted,  
context 
suggest,  
that 
it might be 
possible 
and worthwhile 
to prove 
a 
`quasi-isometric version' 
of Bieberbach's Theorems 
on space groups. 

It would also be 
of interest 
to determine 
whether 
our Main Theorem 
can also be 
derived 
using 
Theorem~1  
in Trofimov's paper~\cite{Gphs_poly-growth}, 
thus avoiding 
the use of 
Losert's result~\cite[Proposition~1]{struc(G-poly_growth)2}.  

\section{Outline 
of the Proof 
of the Main Theorem 
\&\ 
Statement of further Results}\label{sec:proof(MainThm)-outline}

As our first step,  
we use 
a structure result 
on compactly generated, locally compact groups 
of polynomial growth 
due to Losert 
together with 
Gromov's theorem 
on finitely generated groups of polynomial growth 
to show 
Theorem~\ref{thm:lf-polygrow-orbits=>fg-nilp-mod-Stab(cpop<G)} below. 

We 
now recall 
terminology 
used in 
the statement of 
that theorem. 
Let $d>0$. 
A metric space 
is called 
\emph{$d$-connected} 
if and only if 
for any 
ordered pair $(x,y)$ 
of points 
in the space 
one can find 
a finite sequence of points 
beginning with $x$ 
and 
ending with $y$,  
whose consecutive terms 
are 
at most $d$ apart. 
A metric space 
is called 
\emph{coarsely connected} 
if and only if 
it is $d$-connected 
for some 
positive number $d$.  
The image 
of a coarsely connected space 
under 
a quasi-isometric embedding 
as well as 
any $\epsilon$-neighborhood of 
a coarsely connected space 
is itself 
coarsely connected; 
in particular, 
the property 
of being 
coarsely connected 
is invariant under 
quasi-isometry. 
Further, 
call 
a set of automorphisms $B$ 
of a totally disconnected, locally compact group $G$ 
\emph{bounded\/} 
if and only if 
the set $B.V$ 
has bounded diameter 
for some 
(equivalently, every) 
$V$ in $\mathcal{B}(G)$. 
%

%
If 
in Theorem~\ref{thm:lf-polygrow-orbits=>fg-nilp-mod-Stab(cpop<G)} 
the group $\mathcal{H}$ 
is flat 
of finite rank, 
then 
any minimizing subgroup 
for $\mathcal{H}$ 
satisfies 
the conditions  
on the group $O$ 
in that theorem 
by part~3 
of Theorem~\ref{thm:norm>flat-gp}. 
%
Later 
in the paper 
we will 
use this observation. 
In Section~\ref{sec:examples&applications},  
we will give 
examples 
of groups of automorphisms 
whose orbits 
are quasi-flats 
which are 
not flat.

\begin{theorem}[automorphism groups 
with 
a proper, 
coarsely connected orbit 
of polynomial growth 
are 
virtually 
bounded-by-finitely-generated-nilpotent]  
\label{thm:lf-polygrow-orbits=>fg-nilp-mod-Stab(cpop<G)}
~\\
Let $G$ be 
a totally disconnected, locally compact group 
and 
$\mathcal{H}\leqslant \Aut G$.  
Assume that 
there is 
an $O$ in $\mathcal{B}(G)$
such that 
the orbit $\mathcal{H}.O$ 
is 
proper, 
coarsely connected  
and 
of polynomial growth.  
%
Then 
there is 
a subgroup $\mathcal{H}_0$ 
of finite index 
in $\mathcal{H}$, 
$\mathcal{N}_0\normaleq \mathcal{H}_0$ 
and 
$V\in\mathcal{B}(G)$ 
such that 
$\mathcal{N}_0$ stabilizes $V$ 
and 
$\mathcal{H}_0/\mathcal{N}_0$ is 
a finitely generated nilpotent group. 
\end{theorem}

We then apply 
a flatness criterion 
from~\cite{polycyclicGs=virt-flat},  
restated here 
as Theorem~\ref{thm:nilpotent=flat}.

\begin{theorem}[bounded-by-finitely-generated-nilpotent groups are flat]
\label{thm:nilpotent=flat}
~\\
Let $\mathcal{N}_0\normaleq \mathcal{H}_0\leqslant \Aut G$ 
and 
suppose that 
$\mathcal{N}_0$ stabilizes 
some compact, open subgroup $V$ 
of $G$ 
and that 
$\mathcal{H}_0/\mathcal{N}_0$ is a finitely generated nilpotent group. 
Then 
$\mathcal{H}_0$ is flat 
of finite rank. 
\end{theorem}

Combining 
Theorems~\ref{thm:lf-polygrow-orbits=>fg-nilp-mod-Stab(cpop<G)} and~\ref{thm:nilpotent=flat} 
we conclude that 
the subgroup $\mathcal{H}_0$ 
found in Theorem~\ref{thm:lf-polygrow-orbits=>fg-nilp-mod-Stab(cpop<G)} 
is flat. 
Thus, 
we derive 
the following theorem,  
modulo 
the claim 
contained 
therein 
on the rank 
of $\mathcal{H}_0$. 
That theorem 
delivers 
the same conclusion 
as our desired result 
from weaker hypotheses 
and 
will thus 
prove 
our Main Theorem, 
Theorem~\ref{thm:MainThm_weak-form}. 

\begin{theorem}[geometric characterization of flatness if there is a proper orbit]
\label{thm:MainThm}
Let $G$ be 
a totally disconnected, locally compact group 
and 
$\mathcal{H}\leqslant \Aut G$.  
Assume that 
there is 
an $O$ in $\mathcal{B}(G)$ 
such that 
the orbit $\mathcal{H}.O$ 
is 
proper, 
coarsely connected  
and 
of polynomial growth 
of degree $n$ 
but not $n-1$.  
%
Then 
$\mathcal{H}$ 
has a subgroup $\mathcal{H}_0$ 
of finite index 
that is flat 
of rank $n$.  
\end{theorem} 

The strong form 
of our main theorem, 
Theorem~\ref{thm:MainThm}, 
and 
the remark 
in the paragraph 
preceding Theorem~\ref{thm:lf-polygrow-orbits=>fg-nilp-mod-Stab(cpop<G)}
imply 
the following corollary. 

\begin{corollary}
Let $G$ be 
a totally disconnected, locally compact group. 
A group of automorphisms 
of $G$ 
whose orbits 
in $\mathcal{B}(G)$ 
are quasi-flats 
of finite dimension 
is virtually flat 
if and only if 
it has a proper orbit 
in $\mathcal{B}(G)$. 
\end{corollary}

The proof 
of the existence 
of the flat subgroup $\mathcal{H}_0$ 
in Theorem~\ref{thm:MainThm} 
is not constructive. 
Nevertheless, 
$\mathcal{H}$ has 
a maximal normal flat subgroup, 
whose elements 
are characterized 
intrinsically 
in the next theorem.

\begin{theorem}[bounded conjugacy classes max. flat in a virtually flat group]
\label{thm:bounded-conjugacy-classe=max-normal-flat}
Suppose that 
$\mathcal{H}$ is 
a group 
of automorphisms 
of a totally disconnected, locally compact group 
that is 
virtually flat 
of finite rank. 
Then 
the subset of $\mathcal{H}$ 
defined by 
$
\mathcal{H}_{{FC}_d}:=\{\varphi\in\mathcal{H}\colon \varphi^{\mathcal{H}}\ \text{is bounded}\}
$ 
is a flat, 
normal subgroup 
of $\mathcal{H}$ 
of finite index 
that contains 
each flat 
subgroup 
of finite index.  
\end{theorem}

\section{Proof of 
the geometric flatness criterion 
if 
there is 
a proper orbit}
\label{sec:pf(MainThm)}

We follow 
the outline 
of the proof 
given in 
the last section. 

\begin{proof}[of Theorem~\ref{thm:lf-polygrow-orbits=>fg-nilp-mod-Stab(cpop<G)}]
In the first step 
of the proof,  
the assumptions 
on the orbit $\mathcal{H}.O$ 
are used 
to define 
a locally compact 
completion, 
$H$,  
of $\mathcal{H}$.  
Choose $d>0$ 
such that 
$\mathcal{H}.O$ is 
$d$-connected 
and 
let $X$ be 
the metric graph 
whose set of vertices 
is $\mathcal{H}.O$ 
and 
whose edges 
connect 
precisely those 
pairs of points 
in $\mathcal{H}.O$ 
whose 
mutual distance 
is at most $d$. 
The choice 
of $d$ 
guarantees that 
the graph $X$ 
is connected; 
since 
$\mathcal{H}.O$ is proper 
$X$ is also 
locally finite. 
The group $\mathcal{H}$ 
acts 
by graph-automorphisms 
on $X$ 
and 
the induced action 
on the set 
of vertices 
is transitive.

Denote by $H$ 
the closure of 
the group 
of automorphisms 
of $X$ 
induced by $\mathcal{H}$ 
in the group 
of all graph-automorphisms 
of $X$. 
Since 
the graph $X$ 
is connected 
and 
locally finite, 
the group $\Aut X$ 
and 
its closed subgroup $H$ 
are totally disconnected 
and locally compact 
in the topology 
of pointwise convergence; 
indeed, 
the stabilizer 
of any vertex 
is 
a compact, open subgroup, 
which is profinite. 

Furthermore,  
because 
both $\mathcal{H}$ and $H$ 
act transitively 
on the set 
of vertices 
of $X$, 
these groups 
are generated by 
their respective subsets, 
$\mathcal{H}_E$ respectively $H_E$,  
consisting of 
those automorphisms 
that map 
$O$ to 
any of its neighbors 
in $X$. 
%
Now the set 
of elements 
that map $O$ 
to another vertex, 
$O'$ say, 
is a left coset 
of the stabilizer of $O$; 
within 
the group $H$, 
the latter subset 
is compact. 
Since $H_E$ 
is a finite union 
of such sets, 
we conclude that 
it is compact 
also. 
Hence 
the group $H$ 
is totally disconnected, locally compact and compactly generated. 

We claim next 
that 
$H$ has 
polynomial growth. 
This is seen  
as follows. 
The growth 
of $H$ 
with respect to $H_E$ 
equals 
the growth 
of the combinatorial graph $X$. 
For any radius, 
$r$ say, 
the ball 
of radius $r$ 
around $O$ 
in $\mathcal{H}.O$ 
with respect to 
the metric 
of $\mathcal{B}(G)$ 
contains 
the ball 
with radius $r/d$  
with respect to 
the metric 
of $X$. 
Since 
the  former balls 
grow polynomially, 
the latter do also 
and thus 
$H$ has polynomial growth. 

Since 
$H$ has been seen 
to be 
a compactly generated, totally disconnected, locally compact group of polynomial growth, 
we may apply 
\cite[Proposition~1]{struc(G-poly_growth)2}
to $H$,  
to conclude that 
it has 
a maximal compact normal subgroup $C$, 
such that 
$H/C$ is a Lie group. 
As a quotient 
of a totally disconnected, compactly generated group 
of polynomial growth, 
$H/C$ 
is then 
a discrete, finitely generated group 
of polynomial growth. 

By 
Gromov's theorem 
on finitely generated groups of polynomial growth 
(that is, the Main Theorem 
in~\cite{Gs(poly-growth)+expand-maps} 
or 
Corollary~1.6 
in~\cite{new-pf(Gromov'sThm(Gs(polygrowth)))}) 
$H/C$ is virtually nilpotent. 
Let $H_0$ be 
the inverse image 
of a nilpotent subgroup, 
$N$ say,  
of finite index 
in $H/C$ 
under the canonical projection $H\to H/C$,  
let $\mathcal{H}_0$ 
be the inverse image of $H_0\cap \mathcal{H}$ 
under the map $\mathcal{H}\to H$ 
and 
let $\mathcal{N}_0$ be 
the kernel 
of the composite map 
$\mathcal{H}_0\hookrightarrow\mathcal{H}\to H\to H/C$. 

Then 
$\mathcal{H}_0$ has finite index in $\mathcal{H}$, 
$\mathcal{N}_0$ is normal in $\mathcal{H}_0$ 
and $\mathcal{H}_0/\mathcal{N}_0$ 
is a subgroup 
of the finitely generated, nilpotent group $N$. 
Every subgroup of 
a finitely generated nilpotent group 
is finitely generated; 
see, e.g.  \cite[Lemma~2]{rep(fg-nilpGs)}. 
%
Therefore 
$\mathcal{H}_0/\mathcal{N}_0$ 
is a finitely generated nilpotent group. 

The proof 
of Theorem~\ref{thm:lf-polygrow-orbits=>fg-nilp-mod-Stab(cpop<G)} 
will 
therefore 
be complete,  
once we show that 
there exists 
$V\in\mathcal{B}(G)$ 
such that 
$\mathcal{N}_0$ stabilizes $V$. 
To see this, 
we use that 
the group 
of automorphisms    
of the graph $X$ 
induced by 
the subgroup $\mathcal{N}_0$ 
of $\mathcal{H}$ 
is contained 
in $C$ 
by the definition 
of $\mathcal{N}_0$. 
Since 
the group $C$ 
is compact, 
the set 
of images 
of the vertex $O\in X$ 
under $C$ 
is finite. 
The same 
conclusion 
can 
be drawn 
for the subgroup $\mathcal{N}_0$ 
of $C$. 
Therefore 
the subgroup $V:=\bigcap_{n_0\in\mathcal{N}_0} n_0(O)$ 
is a compact, open subgroup 
of $G$. 
By the definition 
of $V$, 
the group $\mathcal{N}_0$ 
stabilizes $V$, 
and 
our proof 
is complete. 
\qed 
\end{proof}

\begin{proof}[of Theorem~\ref{thm:MainThm}] 
As already noted 
in Section~\ref{sec:proof(MainThm)-outline}, 
modulo 
the claim 
contained 
therein 
on the rank 
of $\mathcal{H}_0$,  
Theorem~\ref{thm:MainThm} follows from 
Theorems~\ref{thm:lf-polygrow-orbits=>fg-nilp-mod-Stab(cpop<G)}, 
using Theorem~\ref{thm:nilpotent=flat}. 

To determine 
the rank 
of the subgroup $\mathcal{H}_0$ 
of $\mathcal{H}$ 
we argue 
as follows. 
Applying 
Lemma~\ref{lem:comparison(orbits(Gs(Isoms)))} 
with $\mathbf{B}:=\mathcal{B}(G)$ 
and 
the natural action of $\mathcal{H}$ 
on $\mathbf{B}$ 
we see that 
all orbits 
of $\mathcal{H}$ 
in $\mathcal{B}(G)$ 
have the same 
type of growth, 
which, 
by assumption, 
is polynomial 
of degree~$n$ 
but not 
of degree~$n-1$. 
Since $\mathcal{H}_0$ 
has finite index 
in $\mathcal{H}$, 
the type 
of growth 
of the orbits 
of $\mathcal{H}_0$ 
in $\mathcal{B}(G)$ 
is the same as 
the type 
of growth 
of 
the orbits 
of $\mathcal{H}$ 
in $\mathcal{B}(G)$. 
If we choose 
a minimizing subgroup, 
$V$ say, 
for the flat group $\mathcal{H}_0$, 
whose rank 
is $r$ say, 
then 
Theorem~\ref{thm:norm>flat-gp} 
implies that 
the metric space $\mathcal{H}_0.V$ 
has polynomial growth 
of degree~$r$ 
but not 
of degree~$r-1$. 
We conclude that 
$r$ equals $n$ 
as claimed. 
This concludes 
the proof 
of Theorem~\ref{thm:MainThm}. 
\qed 
\end{proof}

\section{Subgroup of Bounded Conjugacy Classes 
and 
the Proof of Theorem~\ref{thm:bounded-conjugacy-classe=max-normal-flat}}
\label{sec:pf(max-normal-flat-subgroup)}

As already noticed 
in Section~\ref{sec:proof(MainThm)-outline},  
the proof 
of the existence 
of the flat subgroup $\mathcal{H}_0$ 
in Theorem~\ref{thm:MainThm} 
is not constructive. 
To find $\mathcal{H}_0$, 
we relied on 
Proposition~1 in \cite{struc(G-poly_growth)2}
and 
on Gromov's Theorem, 
both of which 
are non-constructive. 

Nevertheless 
Theorem~\ref{thm:bounded-conjugacy-classe=max-normal-flat} 
gives 
a useful characterization 
of 
the maximal, normal flat subgroup of 
a virtually flat group $\mathcal{H}$ 
of automorphisms 
and 
in this section 
we prove 
that theorem. 
The first step 
towards that result 
is the following Lemma. 
Its proof 
is left 
to the reader. 

\begin{lemma}[bornological group structure 
on the automorphism group] 
\label{lem:bornolG-struc(d)}
~\\
Let $G$ be 
a totally disconnected, locally compact group 
and 
$\mathcal{H}$ a group of automorphisms 
of $G$. 
Then 
the collection of 
bounded subsets 
of $\mathcal{H}$ 
is 
a bornological group structure 
on $\mathcal{H}$; 
that is, 
singleton sets are bounded, 
subsets of bounded sets are bounded, 
finite unions of bounded sets are bounded 
and 
products and inverses of bounded sets are bounded. 
\end{lemma}

The following proposition 
is a consequence 
of Lemma~\ref{lem:bornolG-struc(d)}; 
its proof 
is straightforward 
and 
is left 
to the reader. 

\begin{proposition}[bounded conjugacy classes 
form a normal subgroup]
\label{prop:Props(H_FCd)}
~\\
Let $G$ be 
a totally disconnected, locally compact group 
and 
$\mathcal{H}$ a group of automorphisms 
of $G$. 
Then 
the set 
$
\mathcal{H}_{{FC}_d}:=\{\varphi\in\mathcal{H}\colon \varphi^{\mathcal{H}}\ \text{is bounded}\}
$ 
is a normal subgroup 
of $\mathcal{H}$ 
and 
for any such $\mathcal{H}$ 
we have 
${(\mathcal{H}_{{FC}_d})}_{{FC}_d}=\mathcal{H}_{{FC}_d}$. 
\end{proposition}


%
The next proposition 
explains why 
one should expect 
the subgroup 
of bounded conjugacy classes 
of 
a group 
of automorphisms 
to be related to 
flat subgroups.

\begin{proposition}[a flat group 
has 
bounded conjugacy classes]
\label{prop:flatGs-have-bounded-conjugacy-classes} 
~\\
Suppose that 
$\mathcal{H}$ is 
a flat group of automorphisms 
of a totally disconnected, locally compact group. 
Then 
$\mathcal{H}=\mathcal{H}_{{FC}_d}$. 
\end{proposition}
\begin{proof}
Let $\varphi$ be 
an arbitrary element 
of $\mathcal{H}$. 
Choose 
a compact, open subgroup $O$ 
which is 
minimizing for $\mathcal{H}$. 
Then, 
by part~1 
of Theorem~\ref{thm:norm>flat-gp}, 
we have 
\[ 
\varphi^{\mathcal{H}}.O=
{\varphi}\cdot [{\varphi}^{-1},\mathcal{H}].O\subseteq 
{\varphi}\cdot \mathcal{H}(1).O=
{\varphi}\cdot \mathcal{H}_O.O=
{\varphi}(O)\,, 
\]
and 
hence 
$\varphi^{\mathcal{H}}$ 
is 
bounded. 
\qed 
\end{proof}

We now 
turn to 
the proof 
of Theorem~\ref{thm:bounded-conjugacy-classe=max-normal-flat}. 
That proof 
will consist of 
two steps; 
(Step~1)
under the assumptions 
of Theorem~\ref{thm:bounded-conjugacy-classe=max-normal-flat} 
it will be shown 
in Lemma~\ref{lem:1_bounded-conjugacy-classe=max-normal-flat} 
that 
$\mathcal{H}_{{FC}_d}$ 
is a normal subgroup 
of $\mathcal{H}$ 
that 
contains 
every flat 
subgroup 
of finite index; 
%
(Step~2) 
Lemma~\ref{lem:2_bounded-conjugacy-classe=max-normal-flat}  
applied with 
$\mathcal{H}_{\text{\textit{fcd}}}:=\mathcal{H}_{{FC}_d}$
shows that 
$\mathcal{H}_{{FC}_d}$ is flat. 

\begin{lemma}\label{lem:1_bounded-conjugacy-classe=max-normal-flat} 
Suppose that 
$\mathcal{H}$ is 
a group 
of automorphisms 
of a totally disconnected, locally compact group 
that is 
virtually flat. 
Then 
$\mathcal{H}_{{FC}_d}$ 
is a normal subgroup 
of $\mathcal{H}$ 
that 
contains 
every flat 
subgroup 
of finite index. 
In particular, 
$\mathcal{H}_{{FC}_d}$ has finite index 
in $\mathcal{H}$. 
\end{lemma} 
\begin{proof}
%
%
%
%
Taking 
Proposition~\ref{prop:Props(H_FCd)} 
into account, 
we only need to show 
that 
$\mathcal{H}_{{FC}_d}$ contains 
every flat 
subgroup 
of finite index. 
Let $\mathcal{H}_0$ be 
one such subgroup 
of $\mathcal{H}$ 
and 
$\varphi_0$ an arbitrary element 
of $\mathcal{H}_0$. 
The argument 
of Proposition~\ref{prop:flatGs-have-bounded-conjugacy-classes} 
shows that 
the set $\varphi_0^{\mathcal{H}_0}$ 
is 
a bounded subset 
of $\mathcal{H}_0$ 
and hence 
of $\mathcal{H}$. 
Write 
$\mathcal{H}$ as 
$\bigcup_{i=1}^k \alpha_i\mathcal{H}_0$.  
Then 
$
\varphi_0^{\mathcal{H}}=\bigcup_{i=1}^k \alpha_i{\varphi_0^{\mathcal{H}_0}}\alpha_i^{-1}
$ 
is a finite union 
of bounded subsets 
of $\mathcal{H}$, 
and hence 
is bounded. 
We conclude 
that 
$\varphi_0$ is contained 
in $\mathcal{H}_{{FC}_d}$, 
which proves 
the claim. 
\qed 
\end{proof}

\begin{lemma}\label{lem:2_bounded-conjugacy-classe=max-normal-flat}  
Let $G$ be 
a totally disconnected, locally compact group 
and 
$\mathcal{H}_{\text{\textit{fcd}}}\leqslant \Aut G$.  
Assume that 
there is 
an $O$ in $\mathcal{B}(G)$
such that 
the orbit $\mathcal{H}_{\text{\textit{fcd}}}.O$ 
is proper 
and 
coarsely connected  
(this holds e.g. if 
$\mathcal{H}_{\text{\textit{fcd}}}$ is 
virtually flat 
of finite rank). 
Assume further that 
$\mathcal{H}_{\text{\textit{fcd}}}=({\mathcal{H}_{\text{\textit{fcd}}}})_{{FC}_d}$. 
Then 
there is $\mathcal{N}_0\normaleq \mathcal{H}_{\text{\textit{fcd}}}$ 
and 
$V\in\mathcal{B}(G)$ 
such that 
$\mathcal{N}_0$ stabilizes $V$ 
and 
$\mathcal{H}_{\text{\textit{fcd}}}/\mathcal{N}_0$ is 
a finitely generated abelian group. 
In particular, 
$\mathcal{H}_{\text{\textit{fcd}}}$ is flat 
of finite rank. 
\end{lemma} 
\begin{proof}
We first 
verify 
that 
the condition 
that 
$\mathcal{H}_{\text{\textit{fcd}}}$ is 
virtually flat 
of finite rank 
implies that 
there is 
an $O$ in $\mathcal{B}(G)$
such that 
the orbit $\mathcal{H}_{\text{\textit{fcd}}}.O$ 
is proper 
and 
coarsely connected.   
%
Let $\mathcal{H}_f$ be 
a flat subgroup 
of finite index 
in $\mathcal{H}_{\text{\textit{fcd}}}$ 
and 
$O$ a minimizing subgroup 
for $\mathcal{H}_f$. 
By the remark 
in the paragraph 
preceding Theorem~\ref{thm:lf-polygrow-orbits=>fg-nilp-mod-Stab(cpop<G)}, 
the orbit $\mathcal{H}_f.O$ 
is proper  
and 
coarsely connected.   
%
Since 
$\mathcal{H}_f$ has finite index 
in $\mathcal{H}_{\text{\textit{fcd}}}$, 
the analogous statement 
holds for 
the orbit $\mathcal{H}_{\text{\textit{fcd}}}.O$. 

Define 
a graph $X$ 
as in the proof 
of Theorem~\ref{thm:lf-polygrow-orbits=>fg-nilp-mod-Stab(cpop<G)}. 
Denote by $H$ 
the subgroup 
of $\overline{FC}$-elements in 
the closure, 
$\overline{\mathcal{H}_{\text{\textit{fcd}}}}$, 
of the group 
of automorphisms 
of $X$ 
induced by $\mathcal{H}_{\text{\textit{fcd}}}$ 
in the group 
of all graph-automorphisms 
of $X$. 
As in the proof 
of Theorem~\ref{thm:lf-polygrow-orbits=>fg-nilp-mod-Stab(cpop<G)} 
we conclude that 
the group $\overline{\mathcal{H}_{\text{\textit{fcd}}}}$ 
is totally disconnected and locally compact. 
Theorem~2 in~\cite{FC-ele<tdGs&auts(infGphs)} 
implies that 
$H$ is 
a closed subgroup 
of $\overline{\mathcal{H}_{\text{\textit{fcd}}}}$. 
Note that 
by our assumption 
$\mathcal{H}_{\text{\textit{fcd}}}$ equals $({\mathcal{H}_{\text{\textit{fcd}}}})_{{FC}_d}$ 
and by 
the definition of 
the topology on $\Aut X$, 
$H$ contains 
the image 
of $\mathcal{H}_{\text{\textit{fcd}}}$  
in $\Aut X$ 
and hence 
equals $\overline{\mathcal{H}_{\text{\textit{fcd}}}}$. 
Thus 
$H$ is 
is a totally disconnected, locally compact, compactly generated $\overline{FC}$-group. 

Using Theorem~3.20 in~\cite{cpness-conds<topGs}, 
and 
remembering that 
$H$ is totally disconnected, 
we conclude that 
the group $H$ 
has a compact, normal 
subgroup $P$ 
with 
discrete, torsion free, 
finitely generated  
abelian quotient $A$. 

Put 
$\mathcal{H}_0:= \mathcal{H}_{\text{\textit{fcd}}}$ 
and 
let $\mathcal{N}_0$ 
be equal to 
the kernel 
of the composite homomorphism $\mathcal{H}_{\text{\textit{fcd}}}\to H\to A$. 
Then 
$\mathcal{N}_0$ is 
a normal subgroup 
of  $\mathcal{H}_0$
and 
$\mathcal{H}_0/\mathcal{N}_0$ is 
a finitely generated, nilpotent (in fact, abelian) group. 

Using that 
$\mathcal{N}_0$ is contained in 
the compact subgroup $P$ 
of $\Aut X$ 
we conclude 
as in 
the proof 
of Theorem~\ref{thm:lf-polygrow-orbits=>fg-nilp-mod-Stab(cpop<G)} 
that 
there exists 
$V\in\mathcal{B}(G)$ 
such that 
$\mathcal{N}_0$ stabilizes $V$. 
Applying Theorem~\ref{thm:nilpotent=flat} 
we conclude that 
the group $\mathcal{H}_0= \mathcal{H}_{\text{\textit{fcd}}}$ 
is flat 
of finite rank. 
We have shown 
all parts 
of our claim. 
\qed 
\end{proof}

As explained 
in the paragraph 
preceding 
Lemma~\ref{lem:1_bounded-conjugacy-classe=max-normal-flat}, 
Theorem~\ref{thm:bounded-conjugacy-classe=max-normal-flat} follows 
from what 
we have 
just shown. 

From 
Lemmas~\ref{lem:1_bounded-conjugacy-classe=max-normal-flat}  
and~\ref{lem:2_bounded-conjugacy-classe=max-normal-flat} 
we now 
obtain,  
in the `proper case',  
another characterization of 
virtually flat groups 
of automorphisms 
in terms of 
their elements 
with  
bounded conjugacy class. 

\begin{theorem}%
\label{thm:}
Let $G$ be 
a totally disconnected, locally compact group 
and 
let $\mathcal{H}$ 
be a group of automorphisms 
of $G$. 
Assume that 
there is 
an $O$ in $\mathcal{B}(G)$
such that 
$\mathcal{H}.O$ 
is proper 
and 
coarsely connected. 
Then 
the following conditions 
are equivalent: 
\begin{enumerate}
\item 
$\mathcal{H}$ is virtually flat 
of finite rank. 
\item 
$|\mathcal{H}\colon \mathcal{H}_{FC_d}|$ is finite. 
\end{enumerate}
\end{theorem}
\begin{proof}
That 
(1) implies~(2) 
follows from 
Lemma~\ref{lem:1_bounded-conjugacy-classe=max-normal-flat}. 
Assume 
conversely 
that (2)~holds. 
Then 
the subgroup $\mathcal{H}_{FC_d}$ 
of $\mathcal{H}$
satisfies 
the conditions 
on the group $\mathcal{H}_{\text{\textit{fcd}}}$ 
in Lemma~\ref{lem:2_bounded-conjugacy-classe=max-normal-flat}. 
Applying Lemma~\ref{lem:2_bounded-conjugacy-classe=max-normal-flat} 
we conclude that 
the group $\mathcal{H}_{FC_d}$ 
is flat 
of finite rank. 
Since 
$\mathcal{H}_{FC_d}$ is  
a subgroup 
of finite index 
in $\mathcal{H}$ 
by assumption, 
$\mathcal{H}$ is virtually flat 
of finite rank, 
which is 
the content of 
statement~(1). 
%
\qed 
\end{proof}

\section{An Example and an Application}\label{sec:examples&applications}

A group 
of automorphisms 
of a totally disconnected, locally compact group 
whose orbits 
in the space 
of compact, open subgroups 
are quasi-flats 
need not be 
a flat group 
as 
the following example 
illustrates. 
Conjecture~16 
in the paper~\cite{flatrk(AutGs(buildings))} 
is 
therefore 
false  
as stated; 
the first and third author 
of the current paper 
should have known this 
while being involved 
in writing 
the paper~\cite{flatrk(AutGs(buildings))},  
because 
they obtained 
Theorem~\ref{thm:max-flat<Gs(semisimpleGs)}
below 
by different methods 
in unpublished work 
done in 
the year 2000. 

\begin{example}\label{ex:}
Let 
$p$ be 
a prime number 
and 
$G$ 
the additive group 
of the $2$-dimen\-sional vector space 
over $\mathbb{Q}_p$. 
The metric space 
of compact, open subgroups 
of $G$ 
is proper. 
Let $\mathcal{H}$ be 
the group of automorphisms 
of $G$ 
that is 
generated by 
the following two 
linear transformations: 
\[
s_1:=
\left(\begin{array}{cc}0\;\; & -1 \\1\;\; &\hfill 0\end{array}\right)
\qquad
s_2:=
\left(\begin{array}{cc}0\;\; & -p^{-1} \\p\;\; & 0\end{array}\right)=
\left(\begin{array}{cc}0\;\; & -1 \\1\;\; &\hfill 0\end{array}\right)\cdot 
\left(\begin{array}{cc}p\;\; & 0 \\0\;\; & p^{-1}\end{array}\right)=:
s_1u
\]
We will show 
that 
the orbits 
of compact, open subgroups 
under $\mathcal{H}$ 
are quasi-isometric 
to $\mathbb{R}$, 
but 
that 
the group $\mathcal{H}$ 
is not flat. 
\end{example}
\begin{proof}
%
The group $\mathcal{H}$ is 
isomorphic to 
the infinite dihedral group 
with canonical generators 
$s_1$ and~$s_2$. 

The group $\mathcal{T}:=\langle u\rangle$ 
has index~$2$ 
in $\mathcal{H}$. 
Since 
$\mathcal{T}$ is generated 
by one element, 
it is a flat group 
whose rank 
is either~$0$ or~$1$; 
this may be seen 
for example 
by applying 
Lemma~2 in~\cite{direction(aut(tdlcG))} 
with $u:=\alpha$. 
%
The group $\mathcal{T}$ 
has unbounded orbits 
in $\mathcal{B}(G)$; 
hence 
the rank 
of $\mathcal{T}$ 
must be larger than~$0$ 
and hence 
is~$1$. 
Using Theorem~\ref{thm:norm>flat-gp} 
we conclude 
the orbits of 
the group $\mathcal{T}$ 
in $\mathcal{B}(G)$ 
are 
quasi-isometric to $\mathbb{R}$. 
Since $\mathcal{T}$ 
has finite index 
in $\mathcal{H}$ 
the same 
is true of 
the orbits 
of the group $\mathcal{H}$ 
in $\mathcal{B}(G)$, 
as claimed.  

%
We now show 
by contradiction 
that  
the group $\mathcal{H}$ 
is not flat. 
For suppose 
it were. 
Then 
the elements $s_1$ and~$s_2$ 
are contained 
in the subgroup $\mathcal{H}(1)$, 
because 
they have finite order. 
Since 
$s_1$ and $s_2$ 
generate $\mathcal{H}$, 
we then must have 
$\mathcal{H}=\mathcal{H}(1)$, 
the rank 
of $\mathcal{H}$ 
would be~$0$ 
and 
all orbits 
of $\mathcal{H}$ 
in $\mathcal{B}(G)$ 
would be bounded. 
This is a contradiction 
to what 
we have shown 
to be the case 
above 
and 
we conclude that 
$\mathcal{H}$ is not flat, 
as claimed.  
\qed 
\end{proof}

\begin{remark}
The phenomenon 
illustrated in 
the above example 
does not occur 
with groups 
of automorphisms 
whose orbits 
in $\mathcal{B}(G)$ 
are bounded. 
Under this condition 
one need not 
even 
impose 
a condition 
on 
$\mathcal{B}(G)$. 
That is, 
any group 
of automorphisms 
with bounded orbits 
in $\mathcal{B}(G)$ 
is flat 
of rank~$0$. 
This statement 
is 
Proposition~5 in~\cite{direction(aut(tdlcG))}. 
It is 
a consequence of 
the main result 
of Schlichting's paper~\cite{Op:perio-Stab}; 
the latter result 
has been generalized 
in~\cite{sG-close-to-normal} 
and%
~\cite{almost-inv-families}. 

This stronger form 
of the geometric flatness criterion 
in the rank-$0$-case 
has 
played an important role 
in 
the recent study, 
by Y. Shalom and G. Willis, 
of 
almost normal subgroups of 
arithmetic groups via suitable completions. 
\end{remark}

We 
now show 
how 
the results 
of this paper 
lead to 
a remarkably 
simple 
and straightforward proof 
of the following result. 
These notions 
are used 
in its formulation: 
A flat subgroup 
is a subgroup 
that is flat 
as a group 
of inner automorphisms; 
the flat rank 
of a totally disconnected, locally compact group 
is the supremum 
of the ranks 
of its flat subgroups.

\begin{theorem}
\label{thm:max-flat<Gs(semisimpleGs)}
Let 
$k$ be 
a nonarchimedean local field 
and 
$\mathbf{G}$ 
a connected semisimple group 
that is 
defined over $k$. 
Then, 
for any 
maximal $k$-split torus $\mathbf{S}$ 
of $\mathbf{G}$ 
the group $\mathcal{Z}_{\mathbf{G}}(\mathbf{S})(k)$ 
is a maximal flat subgroup 
of $\mathbf{G}(k)$ 
of maximal rank,  
equal to 
the $k$-rank 
of $\mathbf{G}$. 
In particular, 
the group 
$\mathcal{N}_\mathbf{G}(\mathbf{S})(k)$ 
is not flat; 
however 
the orbits 
of this group  
in $\mathcal{B}(\mathbf{G}(k))$ 
are 
quasi-flats 
of dimension $k\operatorname{-rank}(\mathbf{G})=\operatorname{flat-rank}(\mathbf{G}(k))$.   
\end{theorem}
\begin{proof}
It will be 
obvious 
form the proof,  
that 
an analogous result 
may be obtained 
in other cases 
also. 
The action 
of the group $\mathbf{G}(k)$ 
on the Bruhat-Tits building $X(\mathbf{G},k)$ 
of $\mathbf{G}$ 
over $k$ 
will be used.  

%
Let 
$\mathbf{S}$ be 
a maximal $k$-split torus 
of $\mathbf{G}$. 
The apartment $A_\mathbf{S}$ 
associated to $\mathbf{S}$ 
is 
an affine subspace 
of $X(\mathbf{G},k)$
of maximal dimension. 
The group $N:=\mathcal{N}_\mathbf{G}(\mathbf{S})(k)$ 
stabilizes $A_\mathbf{S}$ 
and 
acts there 
by affine isometries. 
The subgroup 
of $N$ 
that acts 
by translations 
on $A_\mathbf{S}$ 
is $T:=\mathcal{Z}_\mathbf{G}(\mathbf{S})(k)$. 
Since 
$X(\mathbf{G},k)$ 
is also 
the building 
defined by 
a saturated $BN$-pair, 
(namely 
$(B,N)$, 
where 
$B$ is 
a suitable Iwahori subgroup) 
the group $N=\mathcal{N}_\mathbf{G}(\mathbf{S})(k)$ 
is 
the full stabilizer 
of $A_\mathbf{S}$ 
in $\mathbf{G}(k)$. 

We next use 
Theorem~7 from~\cite{flatrk(AutGs(buildings))} 
to turn 
the statements 
of the previous paragraph 
into 
statements 
about 
the action 
of $\mathbf{G}(k)$ 
on $\mathcal{B}(\mathbf{G}(k))$. 
%
We conclude that 
the orbits 
of $N$ and $T$ 
in $\mathcal{B}(\mathbf{G}(k))$ 
are 
quasi-flats 
of dimension $k\operatorname{-rank}(\mathbf{G})$. 
That 
$k\operatorname{-rank}(\mathbf{G})=\operatorname{flat-rank}(\mathbf{G}(k))$ 
follows from 
Corollary~19 in~\cite{flatrk(AutGs(buildings))}.  

%
The orbits of 
the stabilizer of 
a chamber 
in $A_\mathbf{S}$ 
under $N$ and $T$ 
are proper; 
in fact, 
the metric space $\mathcal{B}(\mathbf{G}(k))$ 
is proper. 
From either 
Theorem~\ref{thm:MainThm} 
or
Theorem~\ref{thm:MainThm_weak-form} 
we infer that 
$N$ and $T$ 
are virtually flat 
subgroups 
of $\mathbf{G}(k)$ 
of maximal rank. 
We claim 
next 
that 
the subgroup 
of bounded conjugacy classes $N_{{FC}_d}$ 
of $N$ 
equals $T$. 
Applying Theorem~\ref{thm:bounded-conjugacy-classe=max-normal-flat} 
to this statement 
yields that 
$T$ is flat; 
since 
we know that 
$N$ is 
the full stabilizer 
of $A_\mathbf{S}$, 
every 
flat subgroup 
containing $T$ 
must be 
contained in $N$ 
and 
we will 
obtain 
our remaining claims. 

We are 
therefore 
reduced to proving 
that 
$N_{{FC}_d}=T$. 
The group $T$ 
is 
contained in $N_{{FC}_d}$ 
because 
it acts 
by translations 
on $A_\mathbf{S}$ 
and 
any conjugate 
of a translation 
is a translation 
which displaces 
points 
by the same amount 
as the original translation. 
The set $N\smallsetminus T$ 
consists of 
elements 
which act as
reflections or rotations 
on $A_\mathbf{S}$; 
no such element 
can be 
contained in $N_{{FC}_d}$, 
because 
given 
a fixed point 
of $A_\mathbf{S}$ 
that point 
will be displaced 
by 
an arbitrarily large amount 
by 
a conjugate of  
such a transformation 
by a translation 
of sufficiently large displacement. 
Thus 
$N_{{FC}_d}=T$, 
concluding 
our proof. 
\qed 
\end{proof}

It is uncertain 
whether  
one can 
also 
find a minimizing subgroup 
for the group $\mathcal{Z}_{\mathbf{G}}(\mathbf{S})(k)$ 
in the previous result 
along the same lines. 
After all,  
our approach 
avoided 
the task 
of finding 
a minimizing subgroup.

A minimizing subgroup 
for the group  $\mathcal{Z}_{\mathbf{G}}(\mathbf{S})(k)$ 
is known; 
compare 
Theorem~\ref{Z flat} below. 
We omit the proof 
of this result, 
hoping to find 
a simpler approach 
to it 
which 
might also 
apply
in more general situations. 

\begin{theorem}\label{Z flat}
Let 
$k$ be 
a nonarchimedean local field 
and 
$\mathbf{G}$ 
a connected semisimple group 
that is 
defined over $k$. 
Then, 
for any maximal $k$-split torus $\mathbf{S}$ 
of $\mathbf{G}$ 
the stabilizer of 
any chamber 
of the affine apartment 
corresponding to $\mathbf{S}$ 
is minimizing 
for the group $\mathcal{Z}_{\mathbf{G}}(\mathbf{S})(k)$.
\end{theorem}

 
\def\cprime{$'$} \def\cprime{$'$} \def\cprime{$'$} \def\cprime{$'$}

\end{document}